\documentclass[12pt]{amsart}
\usepackage[english]{babel}
\usepackage{txfonts}
\usepackage{fancyhdr}
\usepackage{amsthm,amsmath,amssymb,amsfonts,latexsym,url}

\usepackage{tikz}
\usepackage[all]{xy}
\usetikzlibrary{decorations.markings}
\usetikzlibrary{backgrounds,shapes}
\usepackage{subfig}
  \usepackage{graphicx}
\usepackage{enumerate}
\makeatletter
\let\@@enum@org\@@enum@
\def\@@enum@[#1]{\@@enum@org[\normalfont #1]}
\makeatother
\newtheorem{thm}{Theorem}[section]

\newtheorem{prop}[thm]{Proposition}

\newtheorem*{claim*}{Claim}

\theoremstyle{remark}

\newtheorem*{hint}{Hint}

\theoremstyle{definition}
\newtheorem{defn}[thm]{Definition}
\newtheorem{prob}{Problem}[section]

\newcommand\be{\begin{enumerate}}
\newcommand\ee{\end{enumerate}}
\newcommand\bp{\begin{proof}}
\newcommand\ep{\end{proof}}
\newcommand\bpp{\begin{prop}}
\newcommand\epp{\end{prop}}
\newcommand\bpb{\begin{prob}}
\newcommand\epb{\end{prob}}
\newcommand\bd{\begin{defn}}
\newcommand\ed{\end{defn}}
\newcommand\bh{\begin{hint}}
\newcommand\eh{\end{hint}}

\newcommand\C{\mathbb{C}}

\newcommand\SU{\operatorname{SU}}

\newcommand\Mod{\operatorname{Mod}}

\newcommand\mC{\mathcal{C}}

\def\thetitle{{Irreducibility of quantum representations of mapping class groups with boundary}}
\def\theshorttitle{}
\usepackage{hyperref}

\begin{document}
\raggedbottom
\title[\theshorttitle]\thetitle
\date{\today}


% author information
\author{Thomas Koberda}
\address{Department of Mathematics, University of Virginia, Charlottesville, VA 22904-4137, USA}
\email{thomas.koberda@gmail.com}

% coauthor information
\author{Ramanujan Santharoubane}
\address{Department of Mathematics, University of Virginia, Charlottesville, VA 22904-4137, USA}
\email{ramanujan.santharoubane@gmail.com}

\begin{abstract} We prove that the Witten--Reshetikhin--Turaev $\mathrm{SU}(2)$ quantum representations of mapping class groups are always irreducible in the case of surfaces equipped with colored banded points, provided that at least one banded point is colored by one. We thus generalize a well--known result due to J. Roberts.

\end{abstract}

\maketitle

\section{Introduction}
Since their inception, the irreducibility of TQFT representations of mapping class groups has remained an often intractable question, and it remains open in general for mapping class groups of closed surfaces. In this article, we consider surfaces with a nonzero number of colored boundary components, under the assumption that at least one of them is colored by $1$. We prove that in this case, the TQFT representations of the corresponding mapping class groups are irreducible, without any number--theoretic assumptions on the level of the representation. As is customary in TQFT, we think of the boundary components of a surface as banded points, i.e. embedded oriented copies of the closed unit interval.

\subsection{Statement of the main result}

Let $S=S_{g}$ be a closed and orientable surface of genus $g\geq 0$, let $n\geq 1$ be an integer such that $4 \le 2g+n$, and let $\underline{k} = (k_1, \ldots , k_n)$ be an $n$--tuple of positive integers. We denote by $\big(S,\underline{k}\big)$ the surface $S$ equipped with $n$ banded points colored by $\{k_1, \ldots , k_n\}$. We denote by $\Mod(S,n)$ the group of diffeomorphisms of $S$ fixing the $n$ banded points up to isotopy.

For $p \ge 6$ an even integer, we write $$\rho_p : \Mod(S,n) \to \mathrm{PAut}\big(V_p\big(S,\underline{k}\big)\big)$$ for the quantum representation of $\Mod(S,n)$ arising from Witten-Reshetikhin-Turaev $\mathrm{SU}(2)$--TQFT. In order for the representation to be defined, we will always assume $p\geq\max_j k_j +4$. The following is our main result:

\begin{thm}\label{thm:main}
Let $n\geq 1$ and suppose that at least one of the colors $k_1, \ldots , k_n$ is $1$. Then the representation $\rho_p$ is irreducible for all $p$.
\end{thm}

We state Theorem \ref{thm:main} for its brevity. We actually prove a stronger fact: the restriction of $\rho_p$ to the (banded) point-pushing subgroup of the point colored by $1$ is irreducible.

\subsection{Notes}
Roberts~\cite{Roberts} proved that the quantum $\SU(2)$ representations of closed mapping class groups are irreducible in the case where $p/2$ is prime. The fundamental fact exploited by Roberts is that when $p/2$ is prime, then the Dehn twists act by diagonalizable linear maps. Moreover, if we consider a maximal collection of commuting Dehn twists, then the joint spectrum of their action has no repeated eigenvalues. The closed mapping class group representations were proved to be irreducible by Korinman~\cite{Korinman} in the cases where $p$ is twice the product of two distinct odd primes and where $p$ is twice the square of an odd prime. Korinman also produces examples where the representations are reducible.

Other important examples of reducibility include~\cite[Theorem 7.9]{bhmv}, where it is proved that in the case where $p \equiv 2  \pmod 4$ and all the banded points on the surface are colored by even numbers, the quantum representations are reducible. In~\cite{Andersen}, Andersen and Fjelstad found three exceptional levels where the quantum representations are reducible as well.

The advantage of the present approach is the fact that we do not place any restrictions on $p$, other than the minimal conditions required to define the quantum $\SU(2)$ representation. We really do need banded points on $S$ and at least one of these colored by $1$, since our proof relies on the quantum representations of surface groups as defined by the authors~\cite{KoberdaSantharoubane}. Theorem 7.9 of \cite{bhmv} and the work of Korinman cited above suggest that this is an essential hypothesis, not just an artifact of the proof.

\section{Background}
We retain the notation from the previous section for the remainder of this paper. Let $\underline{x}=\{x_1, \ldots, x_{n}\}$ be the $n$ banded points on $S$ colored by $k_1, \ldots , k_n$ respectively, and let $\hat{S}$ be the surface $S$ minus $n$ disjoint opened discs containing the $x_j$.

\subsection{TQFT vector spaces} \label{subsect:su2 tqft}
In this subsection, we summarize relevant features of TQFT vector spaces, following~\cite{KoberdaSantharoubane} nearly verbatim.
One can define a certain cobordism category $\mC$ of closed surfaces with colored banded points, in which the cobordisms are decorated by uni-trivalent colored banded graphs. The $\SU(2)$--TQFT is a functor $Z_p$ from the category $\mC$ to the category of finite dimensional vector spaces over $\C$. See~\cite{bhmv}.

Let $\mathcal{H}$ be a handlebody with $\partial{\mathcal{H}} = S$, and let $G$ be a uni-trivalent banded graph such that $\mathcal{H}$ retracts to $G$. We suppose that $G$ meets the boundary of $\mathcal{H}$ exactly at the banded points $\underline{x}$ and this intersection consists exactly of the degree one ends of $G$. A \emph{$p$--admissible coloring} of $G$ is an assignment of an integer to each edge of $G$ such that at each degree three vertex $v$ of $G$, the three (non--negative integer) colors $\{a,b,c\}$ coloring edges meeting at $v$ satisfy certain natural compatibility conditions, and where the color of an edge terminating at a banded point $x_i$ must have the color $k_i$. The details of these conditions are not important for this paper, and the interested reader is directed to~\cite{bhmv}.

\iffalse
\begin{enumerate}
\item
$a+b+c\equiv 0\pmod 2$;
\item
$|a-c|\leq b\leq a+c$;
\item
$a+b+c\leq p-4$;
\item
each color lies between $0$ and $\frac{p}{2}-2$;
\item
the color of an edge terminating at a banded point $x_i$ must have the color $k_i$.
\end{enumerate}
\fi

In what follows, $A$ is a $2p$--primitive root of unity, where $p$ is a sufficiently large integer as in the assumptions of Theorem~\ref{thm:main}. To any $p$-admissible coloring $c$ of $G$, there is a canonical way to associate an element of the skein module \[G^c \in S_{A}(\mathcal{H},\big(S,\underline{k}\big)),\] by cabling the edges of $G$ by appropriate Jones-Wenzl idempotents (see \cite[Section 4]{bhmv} for more detail). There is natural surjective map (see \cite[Proposition 1.9]{bhmv})
$$s :  S_{A}\big(\mathcal{H},\big(S,\underline{k}\big)\big) \to  V_p\big(S,\underline{k}\big).$$ It turns out that the images of the vectors associated to $p$--admissible coloring give a finite basis for $ V_p\big(S,\underline{k}\big)$.

The vector space $V_p\big(S,\underline{k}\big)$ is endowed with a natural hermitian form denoted by $\langle  \, . \, , \, . \, \rangle$, and the basis $\{ G^c \, | \, c \, \, \text{is $p$-admissible} \}$ is orthogonal with respect to this hermitian form (see \cite[Theorem 4.11]{bhmv}).

\subsection{The skein algebra of a surface}
 We denote by $\mathcal{S}_A\big(\hat{S}\big)$ the skein algebra of $\hat{S} \times [0,1]$ with complex coefficients. Recall that it is the complex vector space generated by isotopy classes of banded links in the interior of $\hat{S} \times [0,1]$, subject to the following local relations:

\begin{align} \begin{minipage}[c]{1cm}
\includegraphics[scale = 0.09]{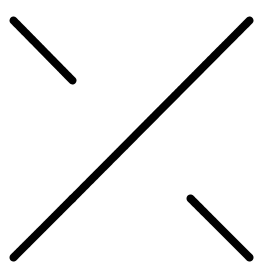} \end{minipage}&= A \, \,\, \begin{minipage}[c]{0.8cm}
\includegraphics[scale = 0.09]{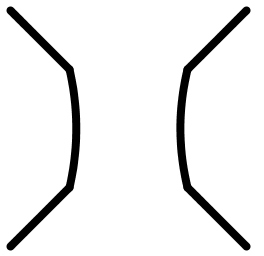} \end{minipage} + A^{-1} \, \,\, \begin{minipage}[c]{1cm}
\includegraphics[scale = 0.09]{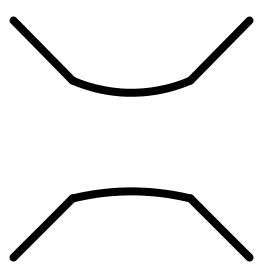} \end{minipage} \notag \\
\notag \\
\begin{minipage}[c]{1cm}
\includegraphics[scale = 0.04]{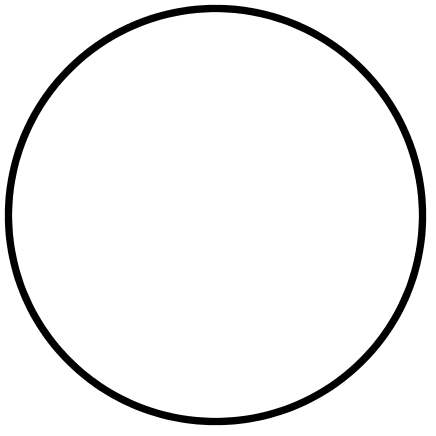} \end{minipage} & = -A^2-A^{-2} \notag 
\end{align}

By stacking banded links, one obtains an algebra structure on $\mathcal{S}_A\big(\hat{S}\big)$.

\subsection{Curve operators}

Let $L$ be a banded link in the interior of $\hat{S} \times [0,1]$. We define the cobordism $C_L$ as $S \times [0,1]$ equipped with the colored banded tangle \[ L \, \cup \, \left(x \times [0,1] \right),\] where here $x = x_1 \cup \cdots \cup x_{n}$ and where each $x_j \times [0,1]$ is colored by $k_j$. By the axioms of the TQFT, one can show that $C_{L}$ defines an operator $Z_p(L) \in \mathrm{End}\big(V_p\big(S,\underline{k}\big)\big)$, called the \emph{curve operator} associated to $L$. Here, $\mathrm{End}\big(V_p\big(S,\underline{k}\big)\big)$ denotes the endomorphisms of $V_p\big(S,\underline{k}\big)$ viewed as a complex vector space, or in other words a matrix algebra over $\mathbb{C}$. It is well known that 
$$Z_p : L \in \mathcal{S}_A\big(\hat{S}\big) \mapsto Z_p(L) \in \mathrm{End}\big(V_p\big(S,\underline{k}\big)\big)$$
is a morphism of algebras.

Any multi-curve (disjoint union of simple closed curves) $\gamma$ can be viewed as an element of $ \mathcal{S}_A\big(\hat{S}\big)$ by associating \[\gamma \mapsto \gamma \times \left[\frac{1}{2},\frac{3}{4}\right] \in  \mathcal{S}_A\big(\hat{S}\big).\] It is well-known that the set of isotopy classes of multi-curves forms a basis for $\mathcal{S}_A\big(\hat{S}\big)$.

\subsection{The Birman Exact Sequence and quantum representations}
We recall the construction of quantum representation of surface groups as carried out by the authors~\cite{KoberdaSantharoubane}. Let $S'$ be the surface $ S \setminus \{x_2, \cdots, x_{n-1} \} $ and let $\Mod(S,n-1)$ be the group of diffeomorphisms of $S$ fixing the $x_2, \cdots, x_{n-1} $ up to isotopy. From the Birman Exact Sequence, the kernel of the canonical map $\Mod(S,n) \to \Mod(S,n-1)$ is a central extension of $\pi_1(S',x_1)$ (See~\cite{Birman,farbmargalit}). Hence by restriction, we have a representation
$$\rho_p : \pi_1(S',x_1) \to \mathrm{PAut}\big(V_p\big(S,\underline{k}\big)\big)$$

Since $x_1$ is a banded point, $\pi_1(S',x_1)$ has to be understood that the fundamental group of $S'$ based at a fixed point chosen on $x_1$.

For the proof of Theorem \ref{thm:main}, it is useful to have a cobordism description of the representation $\rho_p$ restricted to the point pushing subgroup. Let $\gamma : [0,1] \to S'$ be a loop such that $\gamma(0) = \gamma(1)$ is a point on $x_1$. Let $\hat{\gamma}$ be a banded tangle in $S \times [0,1]$ which retracts to $t \in [0,1] \mapsto (\gamma(t),t)$ and which agrees with $x_1 \times \{0,1 \}$ on $S \times \{0,1 \}$. Let $x' = x_2 \cup \cdots \cup x_{n}$. We define the cobordism $\mathcal{C}_{\gamma}$ as $S \times [0,1]$ with the colored banded tangle  $ \hat{\gamma} \, \cup \, \left(x' \times [0,1] \right)$ where each $x_j \times [0,1]$ is colored by $k_j$ and where $\hat{\gamma}$ is colored by $k_1$.

The TQFT functor takes $\mathcal{C}_{\gamma}$ and outputs an operator which is by definition $\rho_p(\gamma)$. We now show how to compute the action of the loop $\gamma$ on the basis of TQFT described in Subsection 
\ref{subsect:su2 tqft}. Let $\mathcal{H}$ be a handlebody with $\partial{\mathcal{H}} = S_g$, let $G$ be a uni-trivalent banded graph such that $\mathcal{H}$ retracts to $G$, and let $G_c$ be a $p$-admissible coloring of this graph. In the TQFT language, applying $\rho_p(\gamma)$ to $G_c$ simply means that we glue the cobordism $C_{\hat{\gamma}}$ to the handlebody $\mathcal{H}$ decorated by $G_c$. Through this gluing, one obtains the same handlebody $\mathcal{H}$, but with a different colored banded graph inside. To express this new banded colored graph in terms of the basis mentioned above, we apply the colored version of the Kauffman bracket and the Jones--Wenzl idempotents. Since the formulas are complicated, we refer the reader to~\cite{MasbaumVogel}.

\section{Proof of Theorem \ref{thm:main}}
Let $\Gamma_p=\rho_p(\pi_1(S',x_1))$ and let $\C[\Gamma_p]$ be the sub algebra of $ \mathrm{End}\big(V_p\big(S,\underline{k}\big)\big)$ generated by $\Gamma_p$.

\begin{prop}\label{prop:image} If $k_1=1$ then $$Z_p\left(\mathcal{S}_A\big(\hat{S}\big)\right) = \C[\Gamma_p]$$
\end{prop}

\begin{proof} Let $\gamma$ be  simple closed curve in $S$ and let \[\hat{\gamma} = \gamma \times \left[\frac{1}{2},\frac{3}{4}\right].\] Let $\mathcal{U} \subset S$ be an open annulus containing $x_1$ and with core curve $\gamma$. Applying the skein relations inside $\mathcal{U} \times [0,1]$, we have the following in $ \mathrm{End}\big(V_p\big(S,\underline{k}\big)\big)$:
$$\begin{minipage}[c]{3.4cm}
\includegraphics[scale = 0.16]{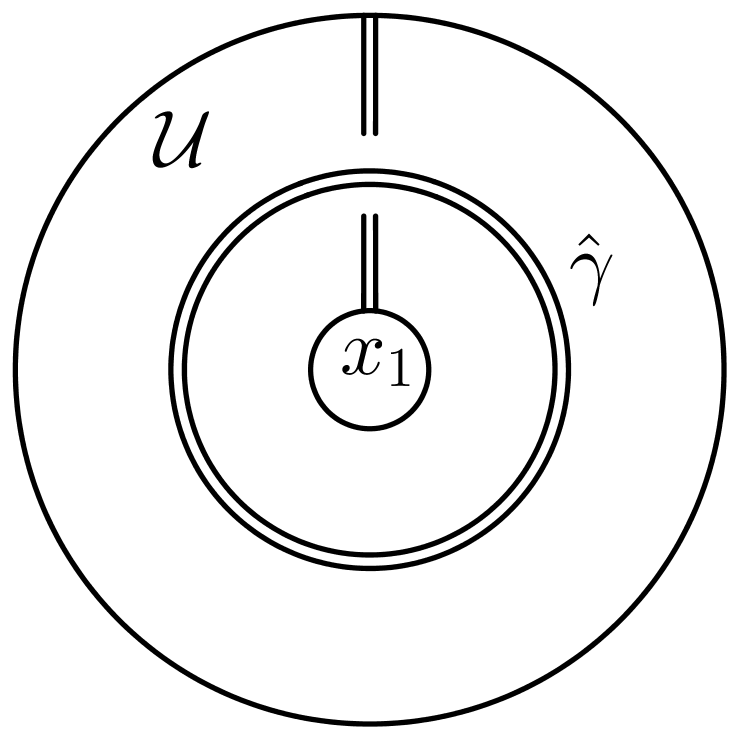} \end{minipage} =A \, \, \begin{minipage}[c]{3.4cm}
\includegraphics[scale = 0.16]{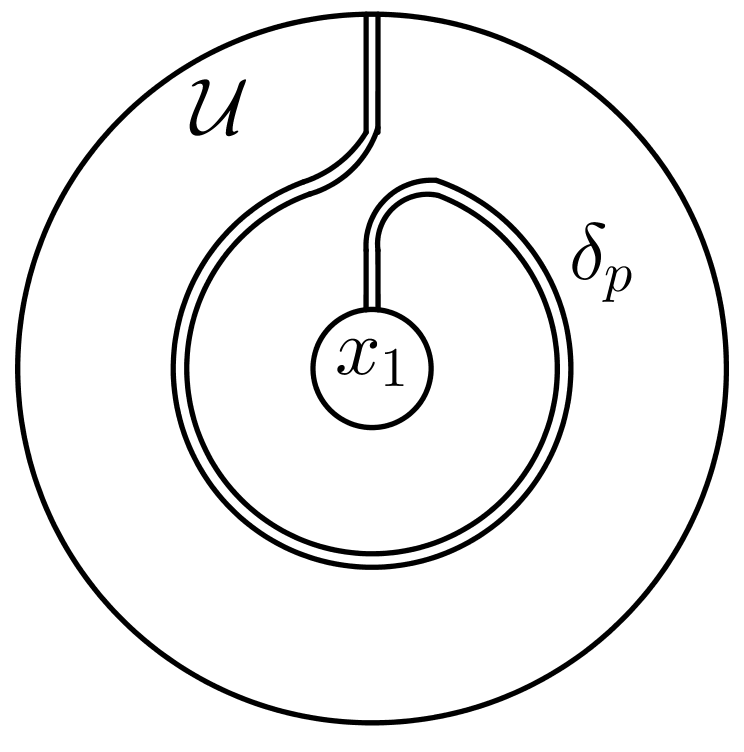} \end{minipage}+A^{-1} \, \, \begin{minipage}[c]{3.4cm}
\includegraphics[scale = 0.16]{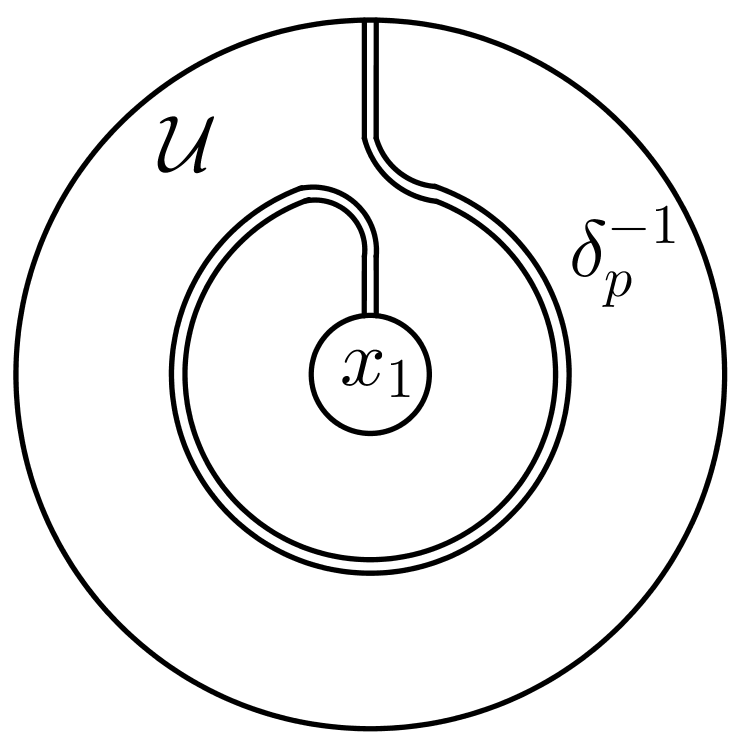} \end{minipage}$$
Observe that the image of $\delta_p$ in $\mathrm{PAut}\big(V_p\big(S,\underline{k}\big)\big)$ is precisely $\rho_p(\delta)$ where $\delta \in \pi_1(S',x_1)$. Hence $$Z_p(\gamma) \in \C[\Gamma_p].$$
This implies that \[Z_p\left(\mathcal{S}_A\big(\hat{S}\big)\right) \subseteq \C[\Gamma_p],\] since the algebra $\mathcal{S}_A\big(\hat{S}\big)$ is generated (as an algebra) by isotopy classes of simple closed curves. The reverse inclusion \[\C[\Gamma_p] \subseteq Z_p\left(\mathcal{S}_A\big(\hat{S}\big)\right)\] is completely general and is established in~\cite{bhmv}.
\end{proof}

In Proposition~\ref{prop:image}, we are unable to comment on the case $k_1>1$. In particular, the analysis of curve operators as carried out in the proof of Proposition~\ref{prop:image} breaks down and we are unable to determine whether the inclusion \[Z_p\left(\mathcal{S}_A\big(\hat{S}\big)\right) \subseteq \C[\Gamma_p]\] remains valid.

The following result can be deduced from Proposition 1.9 of~\cite{bhmv}. The statement given there is somewhat more technical, and we extract a statement more directly applicable to our context, and include a proof for the convenience of the reader.

\begin{prop} \label{prop:irred-skein} Let $\underline{k} = (k_1, \ldots , k_n)$ be arbitrary colors (i.e. without the assumption $k_1=1$).
Then the map $Z_p : \mathcal{S}_A\big(\hat{S}\big) \to \mathrm{End}\big(V_p\big(S,\underline{k}\big)\big)$ is surjective.

\end{prop}
\begin{proof}
Let $\Sigma = \partial(\hat{S} \times [0,1])$ and let $\{\alpha_1, \ldots, \alpha_n\}$ be the boundary curves of $\hat{S}$. These curves can be seen on $\Sigma$ and thus act on $V_p(\Sigma)$ by operators $\{Z_p(\alpha_1), \ldots, Z_p(\alpha_n)\}$.

We set 
$$ E(\underline{k}) = \ker \Big( Z_p(\alpha_1)+A^{2k_1+2}+A^{-2k_1-2} \Big) \cap \cdots \cap \ker \Big( Z_p(\alpha_n)+A^{2k_n+2}+A^{-2k_n-2} \Big). $$ 
An explicit basis for $E(\underline{k})$ can be extracted from the discussion below.

Since $\hat{S} \times [0,1]$ is a handlebody whose boundary is $\Sigma$, it follows from~\cite{bhmv} that the canonical map $$s : \mathcal{S}_A\big(\hat{S}\big) \to V_p(\Sigma)$$ is surjective. Hence the composition of $s$ with the projection to $E(\underline{k})$ is also surjective. 

We claim that $E(\underline{k})$ is canonically isomorphic to $\mathrm{End}\big(V_p\big(S,\underline{k}\big)\big)$. To see this, we suppose that $S$ is embedded in $S^3$ in a way which is unknotted on both sides, and we build the handlebody \[H = S^3 \setminus \hat{S} \times (0,1).\]

Let $G \subset H$ be a banded trivalent graph such that $H$ retracts to $G$. Recall that $\hat{S}$ was obtained from $S$ by removing $n$ discs $\{D_1, \ldots, D_n\}$ about the points $\{x_1,\ldots,x_n\}$. These discs are included in $H$, and cutting $H$ along these discs results in two handlebodies $H_1$ and $H_2$ bounded by $S$. See Figure~\ref{f:proof}.

 \begin{figure}[h]
\includegraphics[scale=0.15]{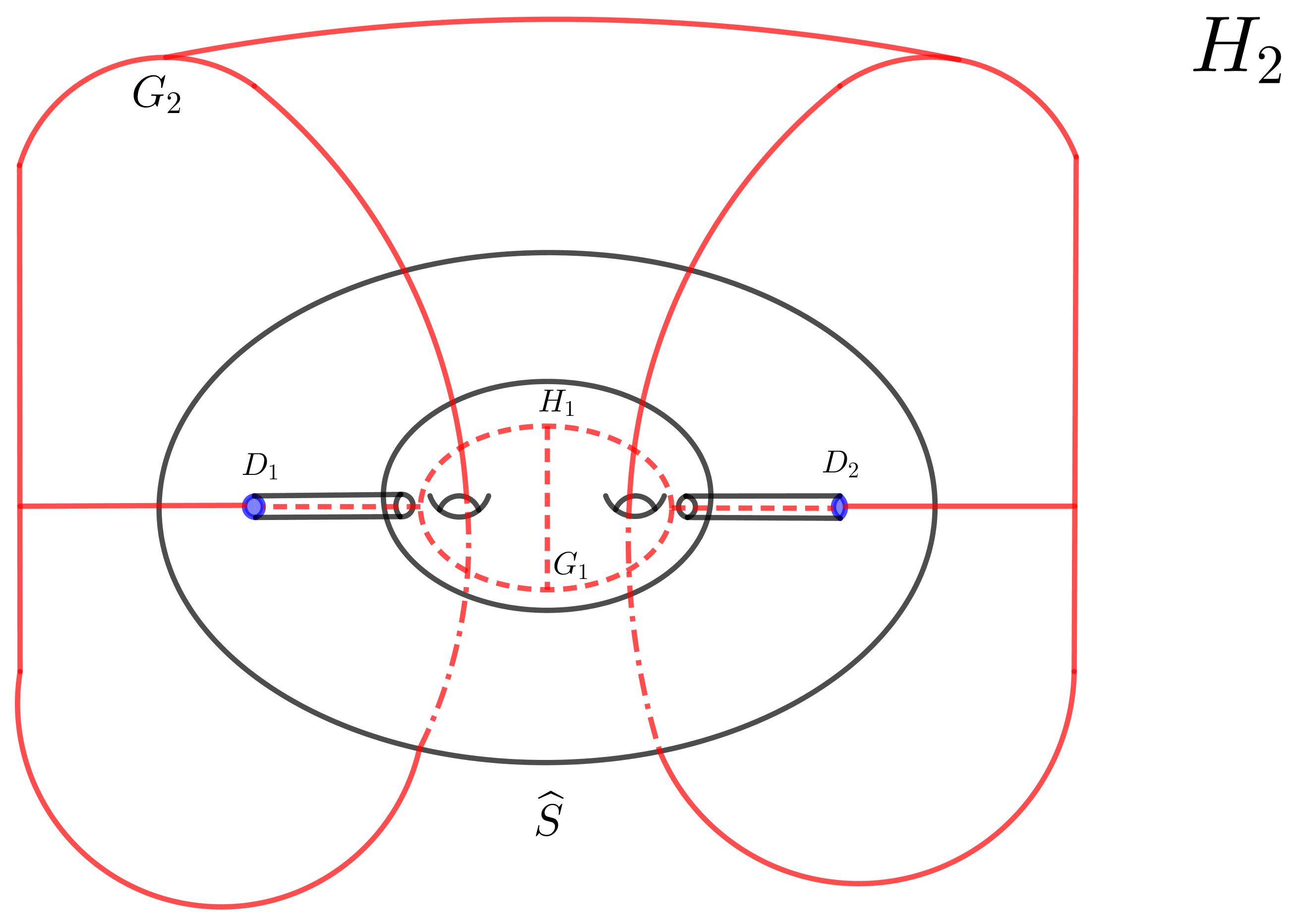}
\caption{This is the setup when $S$ has genus $2$ and when $n=2$. The solid black denotes the surface $\hat{S}\times I$ cut out of $S^3$, which leaves the handlebody $H=H_1\cup H_2$. The handlebodies $H_1$ and $H_2$ meet along discs $D_1$ and $D_2$. The trivalent graph $G=G_1\cup G_2$ (in red) is a retract of $H$. Here, $G_i=G\cap H_i$.}
\label{f:proof}
\end{figure}

We set $G_1 = H_1 \cap G$ and $G_2 = H_2 \cap G$, and let $\Delta(p,\underline{k})$ be the set of $p$-admissible colorings of $G$ such that for $1 \le j \le n$, the edge of $G$ encircled by the curve $\alpha_j$ has color $k_j$. For $c \in \Delta(p,\underline{k})$, we denote by $G^c$ the corresponding colored graph in $V_p(\Sigma)$. Finally, we write $G_1^c$ and $G_2^c$ for the restricted colored graphs. Note that these are naturally elements in $V_p\big(S,\underline{k}\big)$.

Note that $\{ G^c \, | \, c \in \Delta(p,\underline{k}) \}$ is a basis of $E(\underline{k})$, so we can define the linear map 
$$ \phi : G^c \in E(\underline{k}) \mapsto \frac{\langle G^c, G^c \rangle }{\langle G_1^c, G_1^c \rangle  \langle G_2^c, G_2^c \rangle} \, G_1^c \otimes G_2^c \in V_p\big(S,\underline{k}\big)^{\star}\otimes V_p\big(S,\underline{k}\big) = \mathrm{End}\big(V_p\big(S,\underline{k}\big)\big)$$ Here, $\langle  \, . \, , \, . \, \rangle$ is the canonical hermitian form defined on $V_p\big(S,\underline{k}\big)$, and the identification between $V_p\big(S,\underline{k}\big)^{\star}$ and $ V_p\big(S,\underline{k}\big)$ is made via the bilinear pairing. It is not difficult to see that $\phi$ is an isomorphism of vector spaces, because a basis for $E(\underline{k})$ is sent to a basis for $\mathrm{End}\big(V_p\big(S,\underline{k}\big)\big)$.

A straightforward computation using the Hermitian structure shows that the following diagram is commutative:

$$\xymatrix{
     \mathcal{S}_A\big(\hat{S}\big)    \ar[dr]^{Z_p} \ar[r]^{s} & V_p(\Sigma) \ar[r]^{\text{projection}} & E(\underline{k})  \ar[dl]^{\phi}\\
  &  \mathrm{End}\big(V_p\big(S,\underline{k}\big)\big) &  
  }$$

The map $s$ composed with the projection onto $E(\underline{k})$ is surjective, and $\phi$ is an isomorphism. Hence the map $Z_p$ is surjective. 
\end{proof}

\begin{proof}[Proof of Theorem \ref{thm:main}] Combining Proposition \ref{prop:image} and \ref{prop:irred-skein}, we have that if $k_1 = 1$ then $ \C[\Gamma_p] =  \mathrm{End}\big(V_p\big(S,\underline{k}\big)\big) $, and hence $\rho_p$ is irreducible.
\end{proof}

\section{Acknowledgements}
The authors thank J. March\'e and G. Masbaum for helpful discussions. The authors are grateful to the anonymous referee for several helpful comments. The first author is partially supported by Simons Foundation Collaboration Grant number 429836.

\end{document}